\documentclass[12pt,oneside,a4paper,reqno]{amsart}
\usepackage[utf8]{inputenc}
\RequirePackage{fix-cm} % arbitrary font scaling
\usepackage[T1]{fontenc}
\usepackage[english]{babel}
\usepackage{lmodern}
\usepackage{amsmath}
\usepackage{amsthm}
\usepackage{amssymb}
\usepackage{graphicx}
\usepackage{color}
\usepackage{mathtools}
\usepackage{leftidx}
\usepackage{eurosym}
\usepackage{enumerate}
\usepackage{dsfont}

\usepackage{hyperref}	
\usepackage{float}

\usepackage[dvipsnames]{xcolor}
\definecolor{ourcolor}{RGB}{0,102,204}

\usepackage[
   rm={oldstyle=false,proportional=true},
   sf={oldstyle=true,proportional=true},
   tt={oldstyle=true,proportional=true,variable=true},
   qt=false,
 ]{cfr-lm}
 
\usepackage{cite,enumerate,setspace}
\usepackage[normalem]{ulem}
\usepackage{hyperref}
 \hypersetup{
 colorlinks,
 citecolor=ourcolor,
 linkcolor=ourcolor,
 urlcolor=black}

\numberwithin{equation}{section}

\usepackage{geometry}
 \makeatletter
 \def\@seccntformat#1{\hspace*{0mm}%
  \protect\textup{\protect\@secnumfont
    \ifnum\pdfstrcmp{subsection}{#1}=0 \bfseries\fi% subsection # in \bfseries
    \csname the#1\endcsname
    \protect\@secnumpunct
      }%
 }

\newcommand{\KDN}{\kappa(D)}
\newcommand{\N}{\mathbb{N}}

\newcommand{\R}{\mathbb{R}}
\newcommand{\dl}{\mathrm{d}}

\newcommand{\rn}{{\mathbb{R}^{n}}}

\newcommand{\sn}{\mathbb{S}^{n-1}}
\newcommand{\p}{\partial}

\newcommand{\norm}[2]{\|#1\|_{#2}}%Norm
\newcommand{\mv}[2]{\langle#1 \rangle_{#2}}%Mean value

\newcommand{\alp}{\alpha}
\newcommand{\veps}{\varepsilon}
\newcommand{\Om}{\Omega}

\DeclareMathOperator{\diam}{diam}

\DeclareMathOperator{\dist}{dist}

\newtheorem{Theorem}{Theorem}[section]
\newtheorem{example}[Theorem]{Example}
\newtheorem{Corollary}[Theorem]{Corollary}
\newtheorem{Proposition}[Theorem]{Proposition}
\newtheorem{Lemma}[Theorem]{Lemma}
\newtheorem*{Conjecture}{Conjecture}

\theoremstyle{definition}
\newtheorem{Definition}[Theorem]{Definition}
\newtheorem{Remark}[Theorem]{Remark}
\newtheorem{Setting}[Theorem]{Convention}
\newtheorem{Assumption}[Theorem]{Assumption}

\newenvironment{theorem}[1][]
{ \begin{Theorem}[#1]}
	{\end{Theorem}}

\newenvironment{definition}[1][]
{ \begin{Definition}[#1]}
	{\end{Definition} }

\newenvironment{corollary}[1][]
{ \begin{Corollary}[#1]}
	{\end{Corollary} }

\newenvironment{proposition}[1][]
{ \begin{Proposition}[#1]}
	{\end{Proposition}}

\newenvironment{lemma}[1][]
{ \begin{Lemma}[#1]}
	{\end{Lemma}}

\newenvironment{onealign}[1][]
{ \begin{equation}\begin{aligned}[#1]}
{\end{aligned}\end{equation}\ignorespacesafterend}

\numberwithin{equation}{section}
\title[A modular Poincaré--Wirtinger inequality in variable Sobolev spaces]{A modular Poincaré--Wirtinger inequality for Sobolev spaces with variable exponents}
\author[E. Davoli] {Elisa Davoli} 
\author[G. D. Fratta] {Giovanni Di Fratta} 
\author[A. Fiorenza]{Alberto Fiorenza}
\author[L. Happ]{Leon Happ}

\date{}
\AtEndDocument{	
	\hrule
	\medskip
	\small{
		\noindent
		{\sc E. Davoli},
		Institute of Analysis and Scientific Computing,
		TU Wien,
		Wiedner Hauptstra\ss e 8-10, 1040 Vienna, Austria.
		
		\noindent
		E-mail:
		\href{mailto:elisa.davoli@tuwien.ac.at}{\tt elisa.davoli@tuwien.ac.at}.
	}
	
	\medskip \medskip \medskip \hrule  \medskip
	\small{
		\noindent
		{\sc G. Di Fratta},
		Dipartimento di Matematica e Applicazioni ``R. Caccioppoli'',
		Universit\`a degli Studi di Napoli ``Federico II'',
		Via Cintia, Complesso Monte S. Angelo, 80126 Naples, Italy.
		
		\noindent
		E-mail:
		\href{mailto:giovanni.difratta@unina.it}{\tt giovanni.difratta@unina.it}.
	}
	
	\medskip \medskip \medskip \hrule  \medskip
	\small{
		\noindent
		{\sc A. Fiorenza},
		Dipartimento di Architettura,
           Universit\`a degli Studi di Napoli ``Federico II'',
           Via Monteoliveto, I-80134 Napoli, Italy.

          \centerline{\sc and}

        \noindent Istituto per le Applicazioni del Calcolo ``Mauro Picone'', sezione di Napoli, Consiglio Nazionale delle Ricerche, via Pietro Castellino, 111, I-80131 Napoli, Italy.

		\noindent
		E-mail:
		\href{mailto:fiorenza@unina.it}{\tt fiorenza@unina.it}.
	}
   
    \medskip \medskip \medskip \hrule  \medskip
   
    \small{
		\noindent
		{\sc L. Happ},
		Institute of Analysis and Scientific Computing,
		TU Wien,
		Wiedner Hauptstra\ss e 8-10, 1040 Vienna, Austria.

           \centerline{\sc and}

            \noindent MedUni Wien, Währinger Gürtel 18-20, 1090 Vienna, Austria.
		
		\noindent
		E-mail:
		\href{mailto:leon.happ@student.tuwien.ac.at}{\tt leon.happ@student.tuwien.ac.at}.
	}
	\medskip
\hrule
}
\begin{document}
\vskip .2truecm

\begin{abstract}
In the context of Sobolev spaces with variable exponents, Poincaré--Wirtinger inequalities are possible as soon as Luxemburg norms are considered. On the other hand, modular versions of the inequalities in the expected form
\begin{equation*}
        \int_\Omega \left|f(x)-\langle f\rangle_{\Omega}\right|^{p(x)} \ {\mathrm{d} x}
         \leqslant C \int_\Omega|\nabla f(x)|^{p(x)}{\mathrm{d} x},
    \end{equation*}
are known to be \emph{false}. As a result, all available modular versions of the Poincaré- Wirtinger inequality in the variable-exponent setting always contain extra terms that do not disappear in the constant exponent case, preventing such inequalities from reducing to the classical ones in the constant exponent setting. This obstruction is commonly believed to be an unavoidable anomaly of the variable exponent setting.

The main aim of this paper is to show that this is not the case, i.e., that a consistent generalization of the Poincaré-Wirtinger inequality to the variable exponent setting is indeed possible.
Our contribution is threefold. First, we establish that a modular Poincaré--Wirtinger inequality particularizing to the classical one in the constant exponent case is indeed conceivable.
We show that if  \(\Omega\subset \mathbb{R}^n \) is a bounded Lipschitz domain, and if \(p\in L^\infty(\Omega)\), \(p \geqslant 1\), then for every \(f\in C^\infty(\bar\Omega)\) the following generalized Poincaré--Wirtinger inequality holds 
    \begin{equation*}
        \int_\Omega \left|f(x)-\langle f\rangle_{\Omega}\right|^{p(x)} \ {\mathrm{d} x}
         \leqslant C \int_\Omega\int_\Omega \frac{|\nabla f(z)|^{p(x)}}{|z-x|^{n-1}}\ {\mathrm{d} z}{\mathrm{d} x},
    \end{equation*}
    where \(\langle f\rangle_{\Omega}\) denotes the mean of \(f\) over \(\Omega\), and
 \(C>0\) is a positive constant depending only on \(\Omega\) and \(\|p\|_{L^\infty(\Omega)}\).
Second, our argument is concise and constructive and does not rely on compactness results. Third, we additionally provide geometric information on the best Poincaré--Wirtinger constant on Lipschitz domains.

\end{abstract}
\maketitle

%%%%%%%%%%%%%%%%%%%%%%%%%%%%%%%%%%%%%%%%%%
\section{Introduction}
{On the one hand, the} Poincaré--Wirtinger inequality is a fundamental result in mathematical
analysis, for it estimates the deviation of a function from its mean value over a
given bounded domain (namely, an open and connected subset) of $\rn$. In its classical form, the Poincaré--Wirtinger
inequality assures, in particular, that if $\Omega$ is a {\emph{bounded}} Lipschitz domain of
$\rn$, $n  \geqslant 1$, and $p \in [1, \infty)$, then there exists a positive
constant $C > 0$, depending only on $\Omega$ (and hence on $n$) and $p$,
such that
\begin{equation}
  \left( \int_{\Omega} \left| f (x) - \mv{f}{\Om} \right|^p \dl x
  \right)^{1 / p}  \leqslant C \left(
  \int_{\Omega} | \nabla f (z) |^p \ \dl z \right)^{1 / p} \quad \forall f \in
  C^{\infty} (\bar{\Omega}). \label{eq:PoincareIntro}
\end{equation}
Here, and in what follows, for open and bounded \(\Om\subset \rn\) with Lebesgue measure \(|\Om|\) and \(f\in L^1(\Om)\) we denote by 
\begin{equation}
	\mv{f}{\Om}
	:= \frac{1}{|\Om|}\int_\Om f(x) \ {\dl x}
\end{equation}
the mean value of \(f\) in \(\Om\).

{When switching from classical Sobolev spaces to Sobolev spaces with variable exponent, on the other hand, inequalities of the form 
\begin{equation}
\label{eq:false}
        \int_\Omega \left|f(x)-\langle f\rangle_{\Omega}\right|^{p(x)} \ {\mathrm{d} x}
         \leqslant C \int_\Omega|\nabla f(x)|^{p(x)}{\mathrm{d} x},
    \end{equation}
 are known to be false {(cf.~\cite{Fan2005}, the discussion in \cite{Di_Fratta_2022}, { and Example \ref{ex:classic-false} below)}}. Additionally, modular Poincaré--Wirtinger inequalities are generally regarded as only achievable by adding further terms to the right-hand side of \eqref{eq:false} which do not disappear when the exponent $p(x)$ reduces to a constant function (see, e.g., \cite[Theorem~8.2.8 (b), p.257]{Diening2011} and \cite{MR2443740,Ciarlet_2011}). In other words, there is currently no variant of \eqref{eq:PoincareIntro} in the setting of Sobolev spaces with variable exponents only involving modular norms and reducing to \eqref{eq:PoincareIntro} for $p(x)\equiv p\in [1,\infty)$.}

 {
In this paper, we show that such a generalization of \eqref{eq:PoincareIntro} in the setting of Sobolev spaces with variable exponents is indeed possible. Our main result is the following.

\begin{theorem}[Poincaré--Wirtinger-type inequality on Lipschitz domains]\label{thm-pw}
Let \(\Om\subset \rn\) be a bounded Lipschitz domain (open and connected). Let further the variable exponent 
	\begin{equation*}
		p: \Om\to [1,\infty)
	\end{equation*}
	be in \(L^\infty(\Om)\). Then for all \(f\in C^\infty(\overline{\Om})\) there holds
		\begin{equation}\label{eq-pw}
			\int_\Om \big|f(x)-\mv{f}{\Om}\big|^{p(x)} \ {\dl x}
			 \leqslant C \int_\Om\int_\Om \frac{|\nabla f(z)|^{p(x)}}{|z-x|^{n-1}}\ \dl z {\dl x}
		\end{equation}
	for some constant \(C>0\) depending only on \(\Om\) and \(p_+:=\norm{p}{L^\infty(\Om)}\) (see (\ref{eq-final-exp})).
\end{theorem}

In Proposition~\ref{prop-union-ss} (see also Appendix~\ref{aux:appendix}) 
we recall that any bounded Lipschitz domain can be written as
the union of a finite chain of domains that are starshaped with respect to balls all having the same radius.
Precisely, it is always possible 
to express $\Om$ under the form $\Om = \cup_{i=1}^N D_i$ for some \(N\in\N\), where any \(D_i\) is a 
bounded domain that is starshaped with 
respect to some ball with radius \(R>0\) included in  \(D_i\) and $D_{i+1}\cap (\cup_{j=1}^iD_{j}) \ne \emptyset$ 
for all $i=1,...,N-1$.
The constant $C$ in \eqref{eq-pw} can be taken as
\begin{equation}\label{PWConstC1}
C:=\frac{\lambda(\Om)}{\omega_nR^n}\, \cdot \, \left(1+2^{(p_+ +1)}\frac{\kappa(\Om)}{\lambda(\Om)}\right)^N\, \cdot \,  \frac{(\max\{\diam(\Om),1\})^{n+p_+-1}}{n|\Om|} \, ,
\end{equation}
with
\begin{equation}\label{PWConstC2}
	\kappa(\Om):=(n+1)\omega_n \diam(\Om)^n,\qquad 	\lambda(\Om)
	:= \inf_{1 \leqslant i \leqslant N-1}   | D_{i+1}\cap ( \cup_{j=1}^iD_{j} ) |.
\end{equation}
Note that in \eqref{PWConstC1} and \eqref{PWConstC2}, $n$ denotes the dimension of the ambient space $\R^n$, 
while $N$ denotes the number of starshaped domains in which the bounded Lipschitz domain $\Om$ can be decomposed. We used the notation $\omega_n$ to indicate the measure of the unit ball in $\rn$.

As a corollary, we deduce the well-known Poincaré--Wirtinger inequality in classical Sobolev spaces (compare, e.g., with \cite[Remark 13.28, p.433]{Leoni2017}, \cite[Comment 3 on Chapter 9, p.312]{Brezis2011}, \cite[Thm 8.11, p.218]{Lieb2001}, \cite[Proposition 10.2, p.437, for the convex case]{DiBenedetto_2016}, \cite[Corollary 5.4.1, p.173]{Attouch2014}, and \cite[(7.45), p.157]{Gilbarg_2001}):

\begin{corollary} \label{cor:classPW}
	For \(\Om\subset \rn\) a bounded Lipschitz domain (open and connected), \(1 \leqslant p<\infty\) and \(f\in C^\infty(\overline\Om)\), there holds
	\begin{equation}\label{eq:classicalPI}
		\int_\Om \big|f(x)-\mv{f}{\Om}\big|^{p} \ {\dl x}
		 \leqslant \tilde{C} \int_\Om {|\nabla f(z)|^{p}}\ {\dl z}.
	\end{equation}
	for some constant \(\tilde{C}>0\) depending only on \(\Om\) and \(p\).
\end{corollary}

Note that the constant $\tilde{C}$ in \eqref{eq:classicalPI} can be taken as $
\tilde{C}:=C (n+1)\omega_n^{1-\frac{1}{n}}|\Om|^{\frac{1}{n}}$,
with $C$ given by \eqref{PWConstC1}.

	Corollary \ref{cor:classPW} above follows immediately from Theorem \ref{thm-pw}.	If the exponent is constant, we apply Fubini's theorem to the right-hand side of (\ref{eq-pw}) to get
	\begin{align*}
            \int_\Om \big|f(x)-\mv{f}{\Om}\big|^{p} \ {\dl x}
		& \leqslant 
		C\int_\Om\int_\Om \frac{|\nabla f(z)|^{p}}{|z-x|^{n-1}}\ {\dl z}{\dl x}\\
		&=C\int_\Om |\nabla f(z)|^{p}\int_\Om\frac{1}{|z-x|^{n-1}}\ {\dl x} \ {\dl z},
	\end{align*}
        with $C$ given by \eqref{PWConstC1}. The inner integral is now uniformly bounded because \(\Om\) has finite measure. % Now, if $\Omega=\emptyset$ the result is immediate. Otherwise, since $|\Omega|>0$, 
        Applying Lemma \ref{lem-aux-int-bounds} in the appendix, we infer that
	\begin{equation*}
		\int_\Om \big|f(x)-\mv{f}{\Om}\big|^{p} \ {\dl x}
		 \leqslant C (n+1)\omega_n^{1-\frac{1}{n}}|\Om|^{\frac{1}{n}}  \int_\Om |\nabla f(z)|^{p}\ {\dl z},
	\end{equation*}
	from which the result follows with $\tilde{C}:=C (n+1)\omega_n^{1-\frac{1}{n}}|\Om|^{\frac{1}{n}}$.
     %with \(\omega_n\) the measure of the unit ball in \(\R^n\).
}

Inequality \eqref{eq:PoincareIntro} has wide-ranging applications in various fields of
mathematics, including partial differential equations, harmonic analysis,
probability theory, and geometric measure theory. Its significance stems from
its ability to provide powerful estimates about the regularity of functions,
making it a key tool for studying a broad variety of mathematical problems.
Furthermore, the knowledge of sharp constants in such functional inequalities,
%or even upper bounds on the optimal constants that clearly emphasize their
or even upper bounds emphasizing the dependence of optimal constants
 on the domain geometry has remarkable uses in the Analysis of PDEs,
Numerical Analysis, and Materials Science (see, e.g.,
{\cite{Acosta2004,Avkhadiev2007,Dautray1990,DiFratta2019,Fratta_2023,Kuznetsov2015,Michlin1981,Payne1960,Verfuerth1999}}).

The Poincar{\'e}--Wirtinger inequality has been extensively studied over the
years, and numerous extended and refined variants that span various functional
settings have been derived. The literature on the constant-exponent case has
grown so large that a thorough review would exceed the scope of this article.
Instead, we focus on the works which are most relevant to our variable-exponent setting and refer the reader to
{\cite{Brezis2011,Gilbarg_2001,Ruzhansky2019}}, as well as to the comprehensive
presentation in {\cite[Chap.~4]{Ziemer89}}, for more details on
Poincar{\'e}-type inequalities in the constant-exponent case.

The modern argument to derive \eqref{eq:PoincareIntro} proceeds by
contradiction and is based on Rellich--Kondrachov compactness theorem. While
the approach is both elegant and simple, it cannot be used to derive
{\emph{modular}} Poincaré--Wirtinger inequalities in Sobolev spaces with
variable exponent because of a lack of
compactness~{\cite{CruzUribe2014,Diening2011}}.
From \cite[Section~8.2]{Diening2011} we know that in the variable
exponent case, one can state norm Poincaré--Wirtinger inequalities (see, e.g., \cite[Theorem~8.2.4(b), p.255]{Diening2011}, \cite[Lemma~8.2.14, p.260]{Diening2011}, \cite[Theorem~8.2.17, p.262]{Diening2011}; see also \cite{Ciarlet_2011}), but modular versions have always some {\emph{price}} to be paid (see, e.g., \cite[Theorem~8.2.8 (b), p.257]{Diening2011} and \cite{MR2443740,Ciarlet_2011}). This phenomenon appears
typically when working with inequalities in the framework of variable
exponents. See, e.g., {\cite{CruzUribe2018}} and {\cite[Section 3.5,
p.107]{CruzUribe2014}}. We also refer to \cite{Di_Fratta_2022} for a unified treatment of zero-trace Poincar{\'e}
inequality on bounded open subsets and weighted Hardy--Poincar{\'e}
inequalities on unbounded domains.

When dealing with variable exponents, a common technique used to compensate
for the lack of a modular Poincaré--Wirtinger inequality is to
lower-bound the variable exponent with its infimum in some part of the domain
and apply the constant version (see, e.g., {\cite{ouaro2021structural}}), or
to upper-bound the variable exponent with its supremum (see, e.g.,
{\cite{Ayadi_2022,Bendahmane_2010}}), with a consequent loss of information.

For the reasons described above, such strategies were not feasible to prove Theorem \ref{thm-pw} and we had to resort to a different, more classical approach. The proof of Theorem \ref{thm-pw} hinges upon two main basic ingredients. First, the fact that on  domains which are starshaped with respect to an open subset (see Definition \ref{def-ss-i} below) Poincaré--Wirtinger inequalities, even in the setting of Sobolev spaces with variable exponents, can be deduced by a direct, albeit quite delicate, application of the Fundamental Theorem of Calculus. Second, the fact that all Lipschitz domains can be written as finite union of pairwise domains being stricly starshaped with respect to open balls having the same radii. As a byproduct, all proof arguments are coincise and constructive, yielding explicit estimates on the Poincaré--Wirtinger constants. While writing the proof we have kept track of such constants and their dependence on the geometry of the domain
because we believe this will be useful for future possible generalizations.

The paper is organized as follows. Below, in Theorem~\ref{thm-pw}, we state our variable exponent modular Poincaré--Wirtinger inequality;
in Corollary \ref{cor:classPW} we show how to retrieve the classical Poincaré--Wirtinger inequality \eqref{eq:PoincareIntro} from it.
In Section~\ref{sec-prelim}, we introduce some notation and 
results about starshaped and Lipschitz domains.
These results are well-known, but we couldn't locate a single reference collecting all of them. Therefore, in Appendix~\ref{aux:appendix}, we recall their proofs 
using our notational setting. 
The proof of Theorem~\ref{thm-pw} is detailed in Section~\ref{sec:mainsec}.

\section{Preliminary definitions and results}\label{sec-prelim}
Throughout this paper, we say that a subset \(D\subset\rn\) is a \emph{domain} if it is open and connected.

{ We begin this section with an explicit counterexample to \eqref{eq:false}

\begin{example}
\label{ex:classic-false}
{ 
For completeness, we show that \eqref{eq:false} cannot hold. The example mimics the one given in \cite{Fan2005} (see the proof of
Theorem 3.1 therein). Let $\Omega$ be an open set and, to easy notation, assume that $\Omega$ contains the ball centered at the origin of $\mathbb{R}^n$ and of radius $1 + \alpha$, { with $0<\alpha<1$}. Let $\eta \in
C_c^{\infty} (\Omega)$, $0 \leqslant \eta \leqslant 1$, be such that $\text{supp}_{\Omega} \eta \subseteq B_{1 + \alpha}$, $\eta \equiv 1$ in $B_1$, and assume the variable exponent function $p \in L^{\infty} (\Omega)$, $p\geqslant 1$, is radially symmetric and strictly increasing, i.e., that $p (x) = \mathsf{p} (| x |)$ for some strictly increasing function $\mathsf{p} : [0, + \infty) \rightarrow [1, + \infty)$.

For every $0 < \lambda < 1$, we set $f_{\lambda} := \lambda \eta$, and we
compute the corresponding Rayleigh quotient:
\begin{align}
  \frac{\int_{\Omega} | \nabla f_{\lambda} (x) |^{p (x)}  \dl
  x}{\int_{\Omega} | f_{\lambda} (x) - \langle f_{\lambda} \rangle_{\Omega}
  |^{p (x)} \dl x} & = \frac{\int_{B_{1 + \alpha} \setminus B_1}
  \lambda^{p (x)} | \nabla \eta (x) |^{p (x)}  \dl x}{\int_{\Omega}
  \lambda^{p (x)} | \eta (x) - \langle \eta \rangle_{\Omega} |^{p (x)}  \dl
  x} \nonumber\\
  & \leqslant  \frac{\int_{B_{1 + \alpha} \setminus B_1} \lambda^{p (x)} |
  \nabla \eta (x) |^{p (x)}  \dl x}{\int_{B_{1 - \alpha}} \lambda^{p (x)} |
  \eta (x) - \langle \eta \rangle_{\Omega} |^{p (x)}  \dl x} . \nonumber
\end{align}
But on $B_1$ one has \ $0 < | \eta (x) - \langle \eta \rangle_{\Omega} | = | 1
- \langle \eta \rangle_{\Omega} | < 1$ and, therefore continuing estimating
from where we left off, we find that
\begin{align}
  \frac{\int_{\Omega} | \nabla f_{\lambda} (x) |^{p (x)}  \dl
  x}{\int_{\Omega} | f_{\lambda} (x) - \langle f_{\lambda} \rangle_{\Omega}
  |^{p (x)}  \dl x} & \leqslant \frac{(1 + \| \nabla \eta \|_{L^{\infty}
  (\Omega)})^{\mathsf{p} (1 + \alpha)}}{| 1 - \langle \eta \rangle_{\Omega}
  |^{\mathsf{p} (1 - \alpha)}} \cdot \frac{\int_{B_{1 + \alpha} \setminus B_1}
  \lambda^{p (x)}  \dl x}{\int_{B_{1 - \alpha}} \lambda^{p (x)}  \dl x}
  \nonumber\\
  & \leqslant  \frac{(1 + \| \nabla \eta \|_{L^{\infty}
  (\Omega)})^{\mathsf{p} (1 + \alpha)}}{| 1 - \langle \eta \rangle_{\Omega}
  |^{\mathsf{p} (1 - \alpha)}} \cdot \frac{| B_{1 + \alpha} \setminus B_1 |}{|
  B_{1 - \alpha} |} \cdot \frac{\lambda^{\mathsf{p}(1)}}{\lambda^{\mathsf{p} (1 - \alpha)}} . \nonumber
\end{align}
Since $\mathsf{p}(1)>\mathsf{p} (1 - \alpha)$, passing to the limit for $\lambda \rightarrow 0$, we conclude that for the
chosen open set $\Omega$ and radially symmetric and strictly decreasing variable exponent functions, the Rayleigh quotients associated with $f_\lambda$ converge to zero. This implies that the Poincar{\'e}--Wirtinger inequality in the form \eqref{eq:false} cannot hold.}
\end{example}}

One main ingredient for our Poincaré--Wirtinger-type inequality for bounded Lipschitz domains (cf. Theorem \ref{thm-pw}) will make use of the notion of starshaped domains \(D\subset \rn\) with respect to some subset \(S\subset D\). Therefore, in this section we properly introduce them and give some comments on an alternative definition which is, however, equivalent in the setting we will be concerned with. 

\renewcommand{\theequation}{A}
\begin{definition}[Starshaped domains with respect to a subset]\label{def-ss-i}
	A domain \(D\subset \rn\) is called starshaped with respect to a nonempty subset \(S\subset D\) if the following condition holds:
	 \begin{equation}\label{eq-ls}
	 [x,y]\subset D \quad \forall x\in S, \forall y\in D,
	 \end{equation}
	 where \([x,y]\) is the line segment with endpoints \(x\) and \(y\). 
\end{definition}

\renewcommand{\theequation}{\arabic{section}.\arabic{equation}}

We point out that this definition is quite general, since it could be given in any arbitrary vector space and in particular doesn't require any notion of topology. Nevertheless, in the literature also the subsequent definition can be found (we refer to \cite[Def. 2, p.20]{Maz_ya_1998}). 

\begin{definition}\label{def-ss-ii}
 	A domain \(D\subset \rn\) is called starshaped with respect to a nonempty subset \(S\subset D\) if any ray with origin in \(S\) has a unique common point with \(\p D\).
\end{definition}

We did not use a different name to refer to starshaped domains in the sense of Definitions \ref{def-ss-i} and \ref{def-ss-ii} because it turns out that for
the scope of this work we can use both definitions equivalently.
In fact, Definition \ref{def-ss-ii} always implies property (\ref{eq-ls}) and, therefore, a starshaped domain in the sense of Definition \ref{def-ss-ii}
is also starshaped according to  Definition \ref{def-ss-i}.
Viceversa, in the special case that %\(D\) is a domain and 
the subset \(S\subset D\) is open, then a starshaped domain in the sense of Definition \ref{def-ss-i}
is also starshaped according to  Definition \ref{def-ss-ii}.
However, this does not apply in general, making Definition \ref{def-ss-ii} the strongest one. We cover this discussion in detail in Appendix~\ref{aux:appendix} 
(compare Lemma \ref{lem-nc} and Lemma \ref{lem-sc}) along with some simple counterexamples in the general case. 

Hereafter, following the above comments, whenever we use the notion of starshaped sets with respect to a subset, we always refer to Definition \ref{def-ss-i}. 
This is also convenient because condition (\ref{eq-ls}) will prove as the natural assumption for our direct approach. 

Next, we turn to an interesting observation concerning the special geometry of Lipschitz domains. 
The statement can be found for example in \cite[Lemma 1, p.22]{Maz_ya_1998} but with a different notation. For convenience of the reader, we also prove the next result in the appendix~\ref{aux:appendix}.

\begin{proposition}\label{prop-union-ss}
	Any bounded Lipschitz domain \(\Om\subset \rn\) is the union of a finite number of domains that are all starshaped with respect to an open ball of the same radius. 
\end{proposition}

Given a bounded Lipschitz domain $\Om$, Proposition~\ref{prop-union-ss} enables us to prove the Poincaré-Wirtinger-type inequality on each of the starshaped-with-respect-to-a-ball sets forming $\Om$ and then piece together these inequalities. 
The whole argument is detailed in  Section~\ref{sec:mainsec}.

%%%%%%%%%%%%%%%%%%%%%%%%%%%%%%%%%%%%%%%%%%
\section{Proof of Theorem~\ref{thm-pw}} \label{sec:mainsec}
This section is dedicated to the proof of our main theorem. We split it into three steps.

Although for some of the intermediate results we prove it is sufficient to assume that the variable exponent $p: \Om \to [1,\infty]$ is just a measurable function, in what follows we always assume that $p\in L^\infty(\Om)$, and we set $p_+:=\norm{p}{L^\infty(\Om)}$.

The constant \(\omega_n\) will always denote the measure of the unit ball in \(\R^n\). 
We will write \(B_r\) for the open ball with radius \(r> 0\) that is centred at \(0\) and \(B_r(z)\) for the one centred at  \(z\in \rn\). 
Also, we will use Greek letters \(\alp,\beta,\kappa\) to keep track of the constants that appear along the way.

\subsection{Step 1.} \label{sec:step1}
The following lemma provides the crucial estimate, a Morrey-type inequality, for bounded domains which are starshaped with respect to a subset.  

\begin{lemma}\label{lem-morrey-ss}
	Let \(D\subset \rn\) be a bounded domain starshaped with respect to an open subset \(S\subset D\), and let \(f\in C^\infty(\overline{D})\).
	\begin{enumerate}[(i)]
		\item
			For every \(x\in S\) there holds
			\begin{equation}\label{eq-mor1}
				\int_D |f(x)-f(y)|^{p(x)}\ {\dl y}
				 \leqslant \frac{\diam(D)^{n+p(x)-1}}{n+p(x)-1} \int_D \frac{|\nabla f(z)|^{p(x)}}{|z-x|^{n-1}}\ {\dl z}.
			\end{equation} 
		\item
			For every \(x\in D\backslash S\) there holds
			\begin{equation}\label{eq-mor2}
			\int_S|f(x)-f(y)|^{p(x)}\ {\dl y}
			 \leqslant \frac{\diam(D)^{n+p(x)-1}}{n+p(x)-1} \int_D \frac{|\nabla f(z)|^{p(x)}}{|z-x|^{n-1}}\ {\dl z}.
			\end{equation} 
	\end{enumerate}
\end{lemma}
\Remark{Note that although we have the same upper bounds in \eqref{eq-mor1} and \eqref{eq-mor2}, they hold 
for different ranges of the variable $x$. Also, note that the domains of integration on the left-hand sides of \eqref{eq-mor1} and \eqref{eq-mor2} are different.}

\begin{proof}[Proof of Lemma \ref{lem-morrey-ss}]
	{\it (i).}
	Let \(x\in S\) and \(y\in D\).
	We start by noting that for any \(f\in C^\infty(\overline D)\) by the fundamental theorem of calculus there holds
	\begin{equation*}
		f(x)-f(y)
		=-\int_0^1 \nabla f(x+t(y-x))\cdot (y-x)\ {\dl t}.
	\end{equation*}
	Therefore, with the notation \(w(z)=\frac{z}{|z|}\) and by Jensen's inequality, it follows 
	\begin{align*}
		|f(x)-f(y)|^{p(x)}
		& \leqslant \left( \int_0^1 |\nabla f(x+t(y-x))||y-x|\ {\dl t}\right)^{p(x)}\\
		& \leqslant |y-x|^{p(x)}\int_0^1 |\nabla f(x+t(y-x))|^{p(x)}\ {\dl t}\\
		&= |y-x|^{{p(x)}-1}\int_0^{|y-x|} |\nabla f(x+tw(y-x))|^{p(x)}\ {\dl t}.
	\end{align*}
	Integrating the previous relation over \(y\in D\), which is allowed since by assumption \(D\) is starshaped with respect to \(S\), we infer that 
	\begin{onealign}\label{eq-aux1}
		\int_D |f(x)-f(y)|^{{p(x)}}\ {\dl y}
		& \leqslant \int_D |y-x|^{{p(x)}-1}\int_0^{|y-x|} |\nabla f(x+tw(y-x))|^{p(x)}\ {\dl t}{\dl y}\\
		&= \int_\rn |y-x|^{{p(x)}-1}\chi_D(y)\int_0^{|y-x|} |\nabla f(x+tw(y-x))|^{p(x)}\ {\dl t}{\dl y}\\
		&= \int_\rn |y|^{{p(x)}-1}\chi_D(x+y)\int_0^{|y|} |\nabla f(x+tw(y))|^{p(x)}\ {\dl t}{\dl y},
	\end{onealign}
	where \(\chi_D\) is the characteristic function associated with \(D\) that equals \(1\) everywhere inside of \(D\) and vanishes outside of it. 
	Next, we recall that for any function \(u\in L^1(\rn)\), the formula for integration in spherical coordinates yields
        \begin{equation}\label{eq-onion2}
		\int_{\R^n} u(z)\ {\dl z}
		= \int_{0}^\infty r^{n-1}  \int_{\sn}u(r\sigma)\ {\dl \sigma} {\dl r}.
	\end{equation}
	Combining (\ref{eq-aux1}) and (\ref{eq-onion2}) we obtain the estimate
	\begin{equation}\label{eq-aux2}
		\int_D |f(x)-f(y)|^{{p(x)}}\ {\dl y}
		 \leqslant \int_{0}^\infty r^{n+p(x)-2} \int_{\sn}\chi_D(x+r\sigma)\int_0^{r} |\nabla f(x+t\sigma)|^{p(x)}\ {\dl t}{\dl \sigma} {\dl r}.
	\end{equation}
	%The following observation is now a crucial one. 
 Since \(D\) is starshaped with respect to \(S\) and \(x\in S\), we have that for every $r>0$ the following implication holds
	\begin{equation*}
		x+r\sigma\in D
		\implies
		x+t\sigma \in D
		\quad \forall t\in [0,r], \forall \sigma \in \sn \,.
	\end{equation*}
	The previous relation, in terms of characteristic functions, reads as
	\begin{equation}\label{ew-crucial-obs}
		\chi_D(x+r\sigma)
		 \leqslant \chi_D(x+t\sigma)
		\quad  \forall t\in [0,r], \forall \sigma \in \sn.
	\end{equation}
	Hence, it follows from (\ref{eq-aux2}) that 
	\begin{align*}
	\int_D |f(x)-f(y)|^{{p(x)}}\ {\dl y}
	& \leqslant \int_{0}^\infty r^{n+{p(x)}-2}\int_{\sn}\int_0^{r} \chi_D(x+t\sigma)|\nabla f(x+t\sigma)|^{p(x)}\ {\dl t}{\dl \sigma} {\dl r}\\	
	&= \int_{0}^\infty r^{n+{p(x)}-2}\int_0^{r}t^{n-1}\int_{\sn} \chi_D(x+t\sigma)\frac{|\nabla f(x+t\sigma)|^{p(x)}}{|t\sigma|^{n-1}}\ {\dl \sigma} {\dl t}{\dl r}\\
	&= \int_{0}^\infty r^{n+{p(x)}-2}\int_{B_r} \chi_D(x+z)\frac{|\nabla f(x+z)|^{p(x)}}{|z|^{n-1}}\ {\dl z} {\dl r}\\
	&= \int_{0}^\infty r^{n+{p(x)}-2}\int_{B_r(x)\cap D} \frac{|\nabla f(z)|^{p(x)}}{|z-x|^{n-1}}\ {\dl z} {\dl r}.
	\end{align*}
	Since \(x\in D\), the latter term is bounded by
	\begin{equation*}
		\int_{0}^\infty r^{n+{p(x)}-2}\int_{B_r(x)\cap D} \frac{|\nabla f(z)|^{p(x)}}{|z-x|^{n-1}}\ {\dl z} {\dl r}
		 \leqslant \left( \int_{0}^{\diam(D)} r^{n+{p(x)}-2}\ {\dl r}\right)\int_{D}  \frac{|\nabla f(z)|^{p(x)}}{|z-x|^{n-1}}\ {\dl z} .
	\end{equation*}
	Explicitly computing the first integral on the right-hand side completes the proof of (i).
	
	{\it (ii).}
	Since the set \(D\) is starshaped with respect to \(S\) we have that
	\begin{equation*}
		[x,y]\in D\quad \forall x\in D\backslash S, \forall y\in S.
	\end{equation*} 
	Thus, arguing as in the proof of \ref{eq-aux2}, we infer that for every \(x\in D\backslash S\) the following inequality  holds
	\begin{equation}\label{eq-aux3}
		\int_S |f(x)-f(y)|^{{p(x)}}\ {\dl y}
		 \leqslant \int_{0}^\infty r^{n-1}r^{{p(x)}-1} \int_{\sn}\chi_S(x+r\sigma)\int_0^{r} |\nabla f(x+t\sigma)|^{p(x)}\ {\dl t}{\dl \sigma} {\dl r}.
	\end{equation}
	Now, again, for every $r>0$ we have
	\begin{equation*}
		x+r\sigma\in S
		\implies
		x+t\sigma \in D
		\quad \forall t\in [0,r], \forall \sigma \in \sn.
	\end{equation*}
	The previous relation, in terms of characteristic functions, reads as
	\begin{equation}\label{ew-crucial-ob2}
		\chi_S(x+r\sigma)
		 \leqslant \chi_D(x+t\sigma)
		\quad  \forall t\in [0,r], \forall \sigma \in \sn.
	\end{equation}
	We can therewith further estimate (\ref{eq-aux3}) by
	\begin{equation*}
		\int_S |f(x)-f(y)|^{{p(x)}}\ {\dl y}
		 \leqslant \int_{0}^\infty r^{n+{p(x)}-2}\int_{\sn}\int_0^{r} t^{n-1}\chi_D(x+t\sigma)\frac{|\nabla f(x+t\sigma)|^{p(x)}}{|t\sigma|^{n-1}}\ {\dl t}{\dl \sigma} {\dl r} \, ,
	\end{equation*}
	and the claimed inequality follows arguing as in the proof of \eqref{eq-mor1}.
\end{proof}

 Observing that any convex domain is starshaped with respect to itself, we can already prove a modular Poincaré--Wirtinger-type inequality
 for variable exponents and bounded convex domains.

\begin{corollary}\label{cor-pw-convex}
	Let \(\Om\subset \rn\) be a convex, bounded domain. Then for any \(f\in C^\infty(\overline{\Om})\) there holds
	\begin{equation}\label{eq-pw-convex}
	\int_\Om \big|f(x)-\mv{f}{\Om}\big|^{p(x)} \ {\dl x}
	 \leqslant \frac{1}{|\Om|} \int_\Om \frac{\diam(\Om)^{n+p(x)-1}}{n+p(x)-1}  \int_\Om \frac{|\nabla f(z)|^{p(x)}}{|z-x|^{n-1}}\ {\dl z}{\dl x}.
	\end{equation}
\end{corollary}

\begin{proof}
Let \(\Om\subset\rn\) be a convex domain. Then, by definition, for any two points \(x,y\in \Om\) the line segment \([x,y]\) is contained in \(\Om\). 
Following Definition \ref{def-ss-i} this immediately yields that \(\Om\) is starshaped with respect to itself. 
Jensen's inequality (recalling that \(p\ge1\) by assumption) and the previous Lemma \ref{lem-morrey-ss} allow us to infer that 
	\begin{align*}
		\int_\Om \big|f(x)-\mv{f}{\Om}\big|^{p(x)} \ {\dl x}
		&=\int_\Om \bigg|\frac{1}{|\Om|}\int_\Om f(x)- f(y)\ {\dl y}\bigg|^{p(x)} \ {\dl x}\\
		& \leqslant \frac{1}{|\Om|}\int_\Om \int_\Om\big| f(x)- f(y)\big|^{p(x)} \ {\dl x}{\dl y}\\
		& \leqslant\frac{1}{|\Om|} \int_\Om\left(\int_0^{\diam(\Om)} r^{n+p(x)-2}\ {\dl r} \right) \int_\Om \frac{|\nabla f(z)|^{p(x)}}{|z-x|^{n-1}}\ {\dl z}{\dl x}.
	\end{align*}
	This completes the proof.
\end{proof}

If the domain is not convex but still starshaped with respect to some subset, we can deduce a similar estimate. 

\begin{lemma}\label{lem-pw-ss}
Let \(D\subset \rn\) open and bounded. If \(D\) is starshaped with respect to an open subset \(S\subset D\), then for any \(f\in C^\infty(\overline{D})\) the following inequality holds
	\begin{equation}\label{eq-pw-ss}
		\int_{D}\int_D |f(x)-f(y)|^{p(x)} \ {\dl y}{\dl x}
		 \leqslant \alp \int_D\frac{\diam(D)^{n+p(x)-1}}{n+p(x)-1}   \int_D \frac{|\nabla f(z)|^{p(x)}}{|z-x|^{n-1}}\ {\dl z}{\dl x},
	\end{equation}
	where the constant \(\alp>0\) depends only on \(\tilde{p}_+:=\norm{p}{L^\infty(D)}\) and on the sets \(S\) and \(D\). The constant $\alpha$ can be taken as (compare with (\ref{eq-constant-alpha}))
	\begin{onealign}
		\alp(\tilde{p}_+,S,D)
		&:= 2^{\tilde{p}_+}\frac{\kappa(D)}{|S|}.
	\end{onealign}
	with $\kappa(D):=(n+1)\omega_n \diam(D)^n$.
\end{lemma}

\begin{proof}
	By Definition \ref{def-ss-i}, for any \(x,y\in D\) and any \(z\in S\) the line segments \([x,z]\) and \([z,y]\) are completely contained in \(D\). 
	This motivates the following steps. Through the convexity inequality $(a+b)^p  \leqslant 2^{p-1}(a^p+b^p)$, valid for any $a,b \geqslant 0$ and $p \geqslant 1$,
	% \begin{equation}
	% 	(a+b)^p
	% 	 \leqslant 2^{p-1}(a^p+b^p)
	% 	\quad \text{for any } a,b \geqslant 0\text{ and } p \geqslant 1
	% \end{equation} 
	we estimate the left-hand side of (\ref{eq-pw-ss}) observing that for any \(z\in \R^n\) there holds
	\begin{align*}
		&\int_{D}\int_D |f(x)-f(y)|^{p(x)} \ {\dl y}{\dl x}\\
		& \leqslant \int_{D}\int_D 2^{p(x)-1}|f(x)-f(z)|^{p(x)} \ {\dl y}{\dl x} + \int_{D}\int_D 2^{p(x)-1}|f(z)-f(y)|^{p(x)} \ {\dl y}{\dl x}\\
		& \leqslant  |D|\int_{D} 2^{p(x)-1}|f(x)-f(z)|^{p(x)} \ {\dl x} 
		+ \int_{D}2^{p(x)-1}\int_D |f(z)-f(y)|^{p(x)} \ {\dl y}{\dl x}.
	\end{align*}
	Integrating over \(z\in S\) and pulling out the constant \(M(\tilde{p}_+):=2^{\tilde{p}_+-1}\),  yields
	\begin{onealign}\label{eq-first-split}
		&|S|\int_{D}\int_D |f(x)-f(y)|^{p(x)} \ {\dl x}{\dl y}\\
		& \leqslant M(\tilde{p}_+)\left[ |D|\int_{D}\int_S |f(x)-f(z)|^{p(x)} \ {\dl z} {\dl x}
		+ \int_{D} \int_S\int_D|f(z)-f(y)|^{p(x)} \ {\dl y} {\dl z} {\dl x}\right] .
	\end{onealign}
	We further split the first integral on the right-hand side into the two parts
	\begin{onealign}\label{eq-second-split}
		&\int_{D}\int_S |f(x)-f(z)|^{p(x)} \ {\dl z} {\dl x}\\
		&= \int_{S}\int_S|f(x)-f(z)|^{p(x)} \ {\dl z} {\dl x}
		+ \int_{D\backslash S}\int_S |f(x)-f(z)|^{p(x)} \ {\dl z} {\dl x}.
	\end{onealign}
	Since the set \(S\) is starshaped with respect to itself, it is in particular convex, so that by Lemma \ref{lem-morrey-ss} 
	(i) we bound the first contribution in (\ref{eq-second-split}) as
	\begin{onealign}\label{eq-first-part1}
		\int_S\int_S|f(x)-f(z)|^{p(x)} \ {\dl z} {\dl x}
		& \leqslant \int_S \frac{\diam(D)^{n+p(x)-1}}{n+p(x)-1} \int_S \frac{|\nabla f(z)|^{p(x)}}{|z-x|^{n-1}}\ {\dl z}\\
		& \leqslant \int_D \frac{\diam(D)^{n+p(x)-1}}{n+p(x)-1} \int_D \frac{|\nabla f(z)|^{p(x)}}{|z-x|^{n-1}}\ {\dl z}.
	\end{onealign}
	The second contribution of (\ref{eq-second-split}), in turn, is estimated with the help of Lemma \ref{lem-morrey-ss} (ii) to similarly get
	\begin{onealign}\label{eq-first-part2}
		\int_{D\backslash S}\int_S |f(x)-f(z)|^{p(x)} \ {\dl z} {\dl x}
		& \leqslant\int_{D\backslash S}\frac{\diam(D)^{n+p(x)-1}}{n+p(x)-1} \int_D \frac{|\nabla f(w)|^{p(x)}}{|w-x|^{n-1}}\ {\dl w} {\dl x}\\
		& \leqslant \int_{D}\frac{\diam(D)^{n+p(x)-1}}{n+p(x)-1} \int_D \frac{|\nabla f(w)|^{p(x)}}{|w-x|^{n-1}}\ {\dl w} {\dl x}.
	\end{onealign}
	We now turn to the second integral in (\ref{eq-first-split}). We again invoke Lemma \ref{lem-morrey-ss} (i) to get the estimate
	\begin{onealign}\label{eq-second-part-pre}
		\int_{D} \int_S\int_D|f(z)-f(y)|^{p(x)} \ {\dl y} {\dl z} {\dl x}
		& \leqslant \int_D \frac{\diam(D)^{n+p(x)-1}}{n+p(x)-1}\int_S \int_D \frac{|\nabla f(w)|^{p(x)}}{|w-z|^{n-1}}\ {\dl w} {\dl z} {\dl x}\\
        & \leqslant \int_D \frac{\diam(D)^{n+p(x)-1}}{n+p(x)-1}\int_D \int_D \frac{|\nabla f(w)|^{p(x)}}{|w-z|^{n-1}}\ {\dl w} {\dl z} {\dl x}.
	\end{onealign}

We underline the correct use of Lemma~\ref{lem-morrey-ss}~(i): in fact, the double
integral in $\dl y \dl z$ on the left-hand side of \eqref{eq-second-part-pre} is like the left-hand side of \eqref{eq-mor1} provided that one considers $p(x)$ in \eqref{eq-second-part-pre} as a constant
function of the variable $z$.	
	
	%We underline that the exponent in the integral on the left-hand side of Lemma~\ref{lem-morrey-ss}~(i) appears as a function of $z$. While it may not be evident at first glance, this is still the case here, provided one defines $p(z):=p(x)$ for a fixed $x \in D$. Thus, in this instance, replacing $p(z)$ with $p(x)$ does not affect the correct use of Lemma~\ref{lem-morrey-ss}~(i).
	
     From the fact that \(D\) is bounded we infer
	\begin{onealign}\label{eq-swap-vars}
		\int_D \int_D \frac{|\nabla f(w)|^{p(x)}}{|w-z|^{n-1}}\ {\dl w} {\dl z}
		 &=\int_D |w-x|^{n-1}\frac{|\nabla f(w)|^{p(x)}}{|w-x|^{n-1}}\int_D\frac{1}{|w-z|^{n-1}}\ {\dl z}{\dl w}\\
		 & \leqslant \diam(D)^{n-1}\bigg[\sup_{w\in D}\int_D\frac{1}{|w-z|^{n-1}}\ {\dl z}\bigg]
		 \int_D \frac{|\nabla f(w)|^{p(x)}}{|w-x|^{n-1}}\ {\dl w}.
	\end{onealign}
        In view of Lemma \ref{lem-aux-int-bounds} and the estimate \(|D| \leqslant |B_{\diam(D)}|=\omega_n \diam(D)^n\) we obtain the upper bound
        \begin{align*}
            \diam(D)^{n-1}\bigg[\sup_{w\in D}\int_D\frac{1}{|w-z|^{n-1}}\ {\dl z}\bigg]
            & \leqslant (n+1)\diam(D)^{n-1}\omega_n^{1-\frac{1}{n}}|D|^{\frac{1}{n}}\\
            & \leqslant (n+1)\omega_n \diam(D)^n.
        \end{align*}
        Hence, with the definition 
        \begin{equation}\label{eq-constant-kappa}
		\KDN
		:= (n+1)\omega_n \diam(D)^n,
	\end{equation}
        we get from (\ref{eq-swap-vars}) the estimate
        \begin{equation}
            \int_S \int_D \frac{|\nabla f(w)|^{p(x)}}{|w-z|^{n-1}}\ {\dl w} {\dl z}
             \leqslant \KDN\int_D \frac{|\nabla f(w)|^{p(x)}}{|w-x|^{n-1}}\ {\dl w}.
        \end{equation}
	Plugging this into (\ref{eq-second-part-pre}) yields 
	\begin{onealign}\label{eq-second-part}
		&\int_{D} \int_S\int_D|f(z)-f(y)|^{p(x)} \ {\dl y} {\dl z} {\dl x}\\
		& \leqslant \KDN\int_D \frac{\diam(D)^{n+p(x)-1}}{n+p(x)-1}
		\int_D \frac{|\nabla f(w)|^{p(x)}}{|w-x|^{n-1}}\ {\dl w}{\dl x}.
	\end{onealign}
	Finally, by combining (\ref{eq-first-split}) with (\ref{eq-first-part1}), (\ref{eq-first-part2}) and (\ref{eq-second-part}) we conclude that 
	\begin{align*}
		&\int_{D}\int_D |f(x)-f(y)|^{p(x)} \ {\dl y}{\dl x}\\
            & \leqslant \frac{M(\tilde{p}_+)}{|S|}\left( |D| + \KDN\right) \int_D\frac{\diam(D)^{n+p(x)-1}}{n+p(x)-1}
		\int_D \frac{|\nabla f(z)|^{p(x)}}{|z-x|^{n-1}}\ {\dl z}{\dl x}\\
		& \leqslant \alp(\tilde{p}_+,S,D)\int_D \frac{\diam(D)^{n+p(x)-1}}{n+p(x)-1}
		\int_D \frac{|\nabla f(z)|^{p(x)}}{|z-x|^{n-1}}\ {\dl z}{\dl x},
	\end{align*}
	with 
	\begin{onealign}\label{eq-constant-alpha}
		\alp(\tilde{p}_+,S,D)
		&:= 2^{\tilde{p}_+}\frac{\kappa(D)}{|S|}.
	\end{onealign}
 
        In writing the previous estimate we observed that \(|D| \leqslant \kappa(D)\).
\end{proof}

\subsection{Step 2.} \label{sec:step2}
The following auxiliary result will be used to lift inequality (\ref{eq-pw-ss}) to finite unions of bounded domains which are all starshaped with respect to some subsets and  have non-empty intersections.

\begin{lemma}\label{lem-pw-union}
    Let  \(\Om\subset \rn\) be a bounded domain. Let \(D_1, D_2\subset \Om\) be two bounded domains that are starshaped with respect to the open subsets \(S_i\subset D_i\), \(i=1,2\). We further assume that \(D_1\cap D_2\ne \emptyset\). Then, for every \(f\in C^\infty(\overline{\Om})\) there holds:
	\begin{enumerate}[(i)]
		\item
			\begin{equation*}\label{eq-pw-union-i}
				\int_{D_1\times D_2} \big|f(x)-f(y)\big|^{p(x)} \ {\dl y}{\dl x}
				 \leqslant \beta
				\int_{\Om} \frac{\diam(\Om)^{n+p(x)-1}}{n+p(x)-1} \int_{\Om} \frac{|\nabla f(z)|^{p(x)}}{|z-x|^{n-1}}\ {\dl z}{\dl x},
			\end{equation*}
		\item
			\begin{equation*}\label{eq-pw-union-ii}
				\int_{(D_1\cup D_2)\times (D_1\cup D_2)} \big|f(x)-f(y)\big|^{p(x)} \ {\dl y}{\dl x}
				 \leqslant \beta'\int_\Om \frac{\diam(\Om)^{n+p(x)-1}}{n+p(x)-1}  \int_\Om \frac{|\nabla f(z)|^{p(x)}}{|z-x|^{n-1}}\ {\dl z}{\dl x},
			\end{equation*}
	\end{enumerate}
	for constants \(\beta,\beta'>0\) depending on \(p_+:=\norm{p}{L^\infty(\Om)}\), the sets \(S_i, D_i\), \(i=1,2\), and \(\Om\). 
\end{lemma}
\Remark{The proof provides upper bounds for the optimal constants $\beta$ and $\beta'$ (compare (\ref{eq-constant-beta}) and (\ref{eq-constant-beta'})). With $\kappa(\Omega):= (n+1)\omega_n \diam(\Omega)^n$ 
as in \eqref{eq-constant-kappa}, the constants $\beta$ and $\beta'$ take the form
\begin{equation} \label{eq:betas}
	\beta
	:= 2^{p_+-1} \frac{\kappa(\Om)}{|D_1\cap D_2|}(\alp_{\Om,1}+\alp_{\Om,2}),\qquad
	\beta'
	:= \left(2^{p_+} \frac{\kappa(\Om)}{|D_1\cap D_2|}+1\right) (\alp_{\Om,1}+\alp_{\Om,2}) \,	
\end{equation}
where the constants $\alp_{\Om,i}$ are given by  $\alp_{\Om,i}:=2^{p_+}\kappa(\Omega)/|S_i|$.
}
\begin{proof}[Proof of Lemma \ref{lem-pw-union}]
	 Within this proof we will refer to the constants $\alp_i := \alp(p_+,S_i,D_i)$, $i=1,2$, 
	defined in (\ref{eq-constant-alpha}). Following the estimate \(\tilde{p}_+ \leqslant p_+\), we replaced \(\tilde{p}_+\) by \(p_+\). Note that we always have the inequalities
	\begin{equation} \label{eq:alphaOmegai}
		\alp_i
		 \leqslant \alp_{\Om,i}
		:=\alp(p_+,S_i,\Om)
            =\frac{2^{{p_{+}}}}{|S_i|}\kappa(\Omega)
		\quad \text{ for } i=1,2.
	\end{equation}
	We will prove (ii) and will obtain (i) as an intermediate result. We split up the integral in (ii) as
	\begin{align*}
		\int_{(D_1\cup D_2)\times (D_1\cup D_2)} &\big|f(x)-f(y)\big|^{p(x)} \ {\dl y}{\dl x}\\
		= &\int_{D_1\times D_1} \big|f(x)-f(y)\big|^{p(x)} \ {\dl y}{\dl x}
		+ \int_{D_2\times D_2} \big|f(x)-f(y)\big|^{p(x)} \ {\dl y}{\dl x}\\
		&+ \int_{D_1\times D_2} \big|f(x)-f(y)\big|^{p(x)} \ {\dl y}{\dl x}
		+ \int_{D_2\times D_1} \big|f(x)-f(y)\big|^{p(x)} \ {\dl y}{\dl x}\\
		=:&I_1+I_2+I_3+I_4.
	\end{align*}
	We first treat the mixed terms $I_3$ and $I_4$. Since \(D_1\) and \(D_2\) are open with non-empty intersection,
	for any \(z\in B:=D_1\cap D_2\) there holds
	\begin{align*}
		I_3
		&= \int_{D_1}\int_{D_2} \big|f(x)-f(y)\big|^{p(x)} \ {\dl y}{\dl x}\\
		& \leqslant \int_{D_1}\int_{D_2} 2^{p(x)-1}\big|f(x)-f(z)\big|^{p(x)} \ {\dl y}{\dl x}
		+ \int_{D_1}\int_{D_2} 2^{p(x)-1}\big|f(z)-f(y)\big|^{p(x)} \ {\dl y}{\dl x}\\
		& \leqslant M(p_+)\left[ |D_2|\int_{D_1} \big|f(x)-f(z)\big|^{p(x)} \ {\dl x}
		+ \int_{D_1} \int_{D_2}\big|f(z)-f(y)\big|^{p(x)} \ {\dl y}{\dl x}\right],
	\end{align*}
	where \(M(p_+):=2^{p_+-1}\).
	By integrating both sides over \(B\) we find
	\begin{onealign}\label{eq-i3-split}		
		&|B|\int_{D_1}\int_{D_2} \big|f(x)-f(y)\big|^{p(x)} \ {\dl x}{\dl y}\\
		& \leqslant M(p_+)\left[ |D_2|\int_B\int_{D_1} \big|f(x)-f(z)\big|^{p(x)} \ {\dl x}{\dl z}
		+ \int_{D_1}\int_B\int_{D_2} \big|f(z)-f(y)\big|^{p(x)} \ {\dl y}{\dl z}{\dl x}\right] .
	\end{onealign}
       We estimate the first integral on the right-hand side of the previous relation using Lemma \ref{lem-pw-ss}. We get
	\begin{onealign}\label{eq-i31}
		\int_B\int_{D_1} \big|f(x)-f(z)\big|^{p(x)} \ {\dl x}{\dl z}
            & \leqslant \int_{D_1\times D_1} \big|f(x)-f(z)\big|^{p(x)} \ {\dl x}{\dl z}\\
		& \leqslant \alp_1 \int_{D_1}\left(\int_0^{\diam(D_1)} r^{n+p(x)-2}\ {\dl r} \right) \int_{D_1} \frac{|\nabla f(w)|^{p(x)}}{|w-x|^{n-1}}\ {\dl w}{\dl x}\\
		& \leqslant \alp_{\Om,1} \int_{\Om}\left(\int_0^{\diam(\Om)} r^{n+p(x)-2}\ {\dl r} \right) \int_{\Om} \frac{|\nabla f(w)|^{p(x)}}{|w-x|^{n-1}}\ {\dl w}{\dl x}.
	\end{onealign}
        where, we recall,  \(\diam(\Om) \geqslant\diam(D_1)\).
	Further, we observe that Lemma \ref{lem-pw-ss} applies in particular for a constant exponent. Hence, for the second integral on the right-hand side of \eqref{eq-i3-split} we get
    \begin{align*}
        \int_{D_1}&\int_B\int_{D_2} \big|f(z)-f(y)\big|^{p(x)} \ {\dl y}{\dl z}{\dl x}\\
            & \leqslant \int_{\Om}\int_{D_2\times D_2} \big|f(z)-f(y)\big|^{p(x)} \ {\dl y}{\dl z}{\dl x}\\
            & \leqslant \alp_2\int_{\Om}\int_{D_2}\left(\int_0^{\diam(D_2)} r^{n+p(x)-2}\ {\dl r} \right) \int_{D_2} \frac{|\nabla f(u)|^{p(x)}}{|u-z|^{n-1}}\ {\dl u}{\dl z}{\dl x}\\
            &= \alp_2\int_{\Om}\left(\int_0^{\diam(D_2)} r^{n+p(x)-2}\ {\dl r} \right) \int_{D_2}\int_{D_2} \frac{|\nabla f(u)|^{p(x)}}{|u-z|^{n-1}}\ {\dl u}{\dl z}{\dl x}.
    \end{align*}
    With the help of a similar computation to the one already employed in (\ref{eq-swap-vars}) it follows that
	\begin{onealign}\label{eq-i32}
		\int_{D_1}&\int_B\int_{D_2} \big|f(z)-f(y)\big|^{p(x)} \ {\dl y}{\dl z}{\dl x}\\
		& \leqslant  \alp_2\kappa(D_2)\int_{\Om}\left(\int_0^{\diam(D_2)} r^{n+p(x)-2}\ {\dl r} \right) \int_{D_2} \frac{|\nabla f(u)|^{p(x)}}{|u-x|^{n-1}}\ {\dl u}{\dl x}\\
		& \leqslant \alp_{\Om,2}\kappa(\Om) \int_{\Om}\left(\int_0^{\diam(\Om)} r^{n+p(x)-2}\ {\dl r} \right) \int_{\Om} \frac{|\nabla f(u)|^{p(x)}}{|u-x|^{n-1}}\ {\dl u}{\dl x},
	\end{onealign}
	with \(\kappa(D_2),\kappa(\Om)\) given by (\ref{eq-constant-kappa}). Using (\ref{eq-i31}) and (\ref{eq-i32}) to estimate (\ref{eq-i3-split}), and given that $|\Omega| \leqslant\kappa(\Omega)$, we obtain
	\begin{onealign}\label{eq-i3}
		I_3
		& \leqslant \frac{M(p_+)}{|B|}\kappa(\Om)(\alp_{\Om,1}+\alp_{\Om,2}) 
		\int_{\Om}\left(\int_0^{\diam(\Om)} r^{n+p(x)-2}\ {\dl r} \right) \int_{\Om} \frac{|\nabla f(w)|^{p(x)}}{|w-x|^{n-1}}\ {\dl w}{\dl x}.
	\end{onealign}
	This proves (i) with
	\begin{onealign}\label{eq-constant-beta}
	\beta
	&
	:= \frac{2^{p_+-1}}{|D_1\cap D_2|}\kappa(\Om)(\alp_{\Om,1}+\alp_{\Om,2}).
	\end{onealign}
	Note that, by symmetry, the same bound holds for \(I_4\) with the same constant \(\beta\), because the only difference is in the fact that the roles of \(D_1\) and \(D_2\) are exchanged. 
	
	Finally, to estimate $I_1$ and $I_2$ we directly apply Lemma \ref{lem-pw-ss}. Overall, with $\beta$ as in \eqref{eq-constant-beta} and $\alp_{\Om,1},\alp_{\Om,2}$ as in \eqref{eq:alphaOmegai}, we infer (ii) with 
	\begin{equation}\label{eq-constant-beta'}
		\beta'
		:= 2\beta + \alp_{\Om,1} + \alp_{\Om,2}=\left(\frac{2^{p_+}}{|D_1\cap D_2|}\kappa(\Om)+1\right)(\alp_{\Om,1}+\alp_{\Om,2}) \, .
	\end{equation}
	This concludes the proof of the lemma.
\end{proof}
As a direct consequence of the previous proof, we obtain that a similar result holds true for any union of two domains that have a non-empty intersection and each admits some Poincaré--Wirtinger-type inequality as given by Lemma \ref{lem-pw-ss}. 

\begin{corollary}\label{cor-pw-union-arb}
    For a domain \(\Om\subset \rn\), let \(\tilde{D}_1, \tilde{D}_2\subset \Om\) be two bounded domains with \(\tilde{D}_1\cap \tilde{D}_2\ne \emptyset\). Further assume that for any \(f\in C^\infty(\overline{\Om})\) the domains \(\tilde{D}_1, \tilde{D}_2\) admit estimates
	\begin{equation*}\label{eq-pw-ss-arb}
		\int_{\tilde{D}_i\times \tilde{D}_i} |f(x)-f(y)|^{p(x)} \ {\dl y}{\dl x}
		 \leqslant \tilde{\alp}_i \int_{\tilde{D}_i}\frac{\diam(\Omega)^{n+p(x)-1}}{n+p(x)-1} \int_{\tilde{D}_i} \frac{|\nabla f(z)|^{p(x)}}{|z-x|^{n-1}}\ {\dl z}{\dl x}
	\end{equation*}
	for some constants \(\tilde{\alp}_i>0\) for \(i=1,2\). Then, there holds 
	\begin{equation}\label{eq-pw-union-arb}
		\int_{(\tilde{D}_1\cup \tilde{D}_2)\times (\tilde{D}_1\cup \tilde{D}_2)} \big|f(x)-f(y)\big|^{p(x)} \ {\dl y}{\dl x}
		 \leqslant \tilde{\beta}'\int_\Om\frac{\diam(\Omega)^{n+p(x)-1}}{n+p(x)-1} \int_\Om \frac{|\nabla f(z)|^{p(x)}}{|z-x|^{n-1}}\ {\dl z}{\dl x},
	\end{equation}
	where the constant \(\tilde{\beta}'>0\) is given by
	\begin{onealign}\label{eq-constant-betat'}
		\tilde{\beta}'
		:=\left(2^{p_+} \frac{\kappa(\Om)}{|\tilde{D}_1\cap \tilde{D}_2|}+1\right)(\tilde{\alp}_1+\tilde{\alp}_2).
	\end{onealign}
	with \(\kappa(\Om)\) as in (\ref{eq-constant-kappa}).

\end{corollary}
\begin{proof}
	By splitting the integral on the left-hand side of (\ref{eq-pw-union-arb}) as in the proof of Lemma \ref{lem-pw-ss} and 
	following the same subsequent steps, we end up with an inequality as in (\ref{eq-pw-union-arb}) with a constant $\tilde{\beta}'$ (compare with (\ref{eq-constant-beta'})) 
	given by $\tilde{\beta}' = 2 \tilde{\beta}+ \tilde{\alp}_{1} + \tilde{\alp}_{2}$, where 
	\begin{equation*}
		\tilde{\beta}
		:= 2^{p_+-1}\frac{\kappa(\Om)}{|\tilde{D}_1\cap \tilde{D}_2|} (\tilde{\alp}_1+\tilde{\alp}_2).
	\end{equation*}
	The claimed assertion follows.
\end{proof}

\subsection{Step 3.} \label{sec:step3}
We are now in the position to prove our main result.

\renewcommand{\proofname}{Proof of Theorem \ref{thm-pw}}
\begin{proof}
	In view of Proposition \ref{prop-union-ss} we write \(\Om\) as the finite union
	\begin{equation}\label{eq-om-union}
		\Om
		= \bigcup_{i=1}^N D_i
	\end{equation}
    for some \(N\in\N\), where every \(D_i\) is a bounded domain which is starshaped with respect to some subset \(S_i\subset D_i\), where every \(S_i\) can be chosen as a ball with some radius \(R>0\) independent of \(i\). We further note that, without loss of generality, the sets \(D_i\) are labelled so that
	\begin{equation}\label{eq-imp-prop}
		D_{i+1}\cap \left( \bigcup_{j=1}^iD_{j}\right) \ne \emptyset 
		\quad \forall i=1,...,N-1.
	\end{equation}
	To see this, we proceed as follows: We pick an arbitrary \(i_1\in \{1,...,N\}\) and a corresponding \(D_{i_1}\). Since \(\Om\) is open and connected by assumption and all sets \(D_i\) are open as well, there exists \(i_2\ne i_1\) in \(\{1,...,N\}\) such that \(D_{i_2}\cap D_{i_1}\) is non-empty. Now the union \(D_{i_1}\cup D_{i_2}\) is again open and by the same arguments as before there must exist some \(i_3\in \{1,...,N\}\backslash\{i_1,i_2\}\) such that \((D_{i_1}\cup D_{i_2})\cap D_{i_3}\ne \emptyset\). Then arguing iteratively, we find sets \(D_{i_1},...,D_{i_N}\) that exactly fulfil property (\ref{eq-imp-prop}).
	
	We will now inductively apply Corollary \ref{cor-pw-union-arb}. To that end, we first note that according to Lemma \ref{lem-pw-ss} and since \(|S_i|=\omega_nR^n\) for every \(i=1,...,N\), we have the estimates
	\begin{equation*}
		\int_{D_i}\int_{D_i} |f(x)-f(y)|^{p(x)} \ {\dl y}{\dl x}
		 \leqslant \tilde{\alp} \int_\Om \frac{\diam(D_i)^{n+p(x)-1}}{n+p(x)-1}  \int_\Om \frac{|\nabla f(z)|^{p(x)}}{|z-x|^{n-1}}\ {\dl z}{\dl x}
	\end{equation*}
	for
	\begin{equation*}
		\tilde{\alp}
		:= 2^{p_+} \frac{\kappa(\Om)}{\omega_n R^n }.
	\end{equation*}
	We further set 
	\begin{align*}
		\lambda(\Om)
		:= \inf_{1 \leqslant i \leqslant N-1} \textstyle  \left| D_{i+1}\cap \left( \bigcup_{j=1}^iD_{j}\right)\right|.
	\end{align*}
	Then, for \(E_{i}:=D_{i+1}\cup \left( \bigcup_{j=1}^iD_{j}\right)\) we get from Corollary \ref{cor-pw-union-arb} for every \(i\in \{1,...,N-1\}\) an estimate of the form
	\begin{equation*}
	\int_{E_{i}\times E_{i}} \big|f(x)-f(y)\big|^{p(x)} \ {\dl y}{\dl x}
	 \leqslant \tilde{\beta}'_{i}\int_\Om\left(\int_0^\rho r^{n+p(x)-2}\ {\dl r} \right) \int_\Om \frac{|\nabla f(z)|^{p(x)}}{|z-x|^{n-1}}\ {\dl z}{\dl x},
	\end{equation*}
	where the constants \(\tilde{\beta}'_{i}>0\) are recursively given by (as usual we set \(p_+:=\norm{p}{L^\infty(\Om)}\))
	\begin{equation*}
		\tilde{\beta}'_{i}
		:= \left(\kappa_{p_{+}}(\Om) +1 \right)(\tilde{\beta}'_{i-1}+\tilde{\alp})
		\quad \text{ for } i=2,...,N-1
	\end{equation*}
     and $\tilde{\beta}'_{1}
	:= \left(\kappa_{p_{+}}(\Om) +1 \right)2\tilde{\alp}$, with $\kappa_{p_{+}}(\Om)=2^{p_+}{\kappa(\Om)}/{\lambda(\Om)}$. This is a recurrence equation, 
	whose unique solution is given by
	\begin{equation} \label{eq:solrec}
		\tilde{\beta}'_{i}=\frac{\tilde\alpha}{\kappa_{p_{+}}(\Om)} \left((2 \kappa_{p_{+}}(\Om) +1) (\kappa_{p_{+}}(\Om)+1)^i-(\kappa_{p_{+}}(\Om)+1)\right) 
	\end{equation}
	for \(i\in \{1,...,N-1\}\).
	In particular, we have
	\begin{equation} \label{eq:solrec2}
		\tilde{\beta}'_{N-1} \leqslant  \frac{\tilde\alpha}{\kappa_{p_{+}}(\Om)}(1+2\kappa_{p_{+}}(\Om))^{N}\, .
	\end{equation}
	By definition we have \(\Om=E_{N-1}\), and thus
	\begin{equation*}
		\int_{\Om\times \Om} \big|f(x)-f(y)\big|^{p(x)} \ {\dl y}{\dl x}
		 \leqslant \tilde{C}\int_\Om\left(\int_0^\rho r^{n+p(x)-2}\ {\dl r} \right) \int_\Om \frac{|\nabla f(z)|^{p(x)}}{|z-x|^{n-1}}\ {\dl z}{\dl x},
	\end{equation*}
	where the constant \(\tilde{C}=\frac{\tilde\alpha}{\kappa_{p_{+}}(\Om)}(1+2\kappa_{p_{+}}(\Om))^{N}\) is given by the right-hand side of \eqref{eq:solrec2}.
        Thus, by Jensen's inequality 
        \begin{onealign}\label{eq-final-exp}
            \int_\Om \big|f(x)-\mv{f}{\Om}\big|^{p(x)} \ {\dl x}
		& \leqslant \frac{1}{|\Om|}\int_\Om \int_\Om\big| f(x)- f(y)             \big|^{p(x)} \ {\dl x}{\dl y}\\
            & \leqslant \frac{\tilde{C}}{|\Om|}\int_{\Om}\frac{\diam(\Om)^{n+p(x)-1}}{n+p(x)-1}\int_\Om \frac{|\nabla f(z)|^{p(x)}}{|z-x|^{n-1}}\ {\dl z}{\dl x}\\
            & \leqslant \frac{\tilde{C}}{|\Om|}\frac{(\max\{\diam(\Om),1\})^{n+p_+-1}}{n}\int_\Om \int_\Om \frac{|\nabla f(z)|^{p(x)}}{|z-x|^{n-1}}\ {\dl z}{\dl x}.
        \end{onealign}
		This concludes the proof.
\end{proof}
\renewcommand{\proofname}{Proof}

%%%%%%%%%%%%%%%%%%%%%%%%%%%%%%%%%%%%%%%%%%
\appendix
\section{Auxiliary Results} \label{aux:appendix}
\renewcommand{\theequation}{\thesection.\arabic{equation}}
In this appendix, we collect the proofs of the results stated in Section \ref{sec-prelim}, with the goal to clarify the connection between Definitions \ref{def-ss-i} and \ref{def-ss-ii} of starshaped domains with respect to a subset. Afterward, we prove Proposition \ref{prop-union-ss}, showing  that any Lipschitz domain can be written as the finite union of domains which are starshaped with respect to open balls. 
 
\begin{lemma}[Necessity of condition (\ref{eq-ls})]\label{lem-nc}
	Let \(D\subset \rn\) be starshaped with respect to a set \(S\subset D\) according to Definition \ref{def-ss-ii}. Then, it necessarily fulfills condition (\ref{eq-ls}), i.e.
	\begin{equation*}
	[x,y]\subset D \quad \forall x\in S, \forall y\in D,
	\end{equation*}
	where \([x,y]\) is the line segment with endpoints \(x\) and \(y\). 
\end{lemma}
\begin{proof}
	We define the signed distance function associated with \(D\) as
	\begin{equation}\label{eq-signed-dist}
	b_{D}:\rn\to \R, \quad 
	x\mapsto
	\begin{cases}
	\dist(\p D, x),\quad &\text{for\ }  x\in \text{int}(\rn\backslash D),\\
	0, \quad &\text{for\ } x\in \p D,\\
	-\dist(\p D, x),\quad &\text{for }  x\in \text{int}(D).
	\end{cases}
	\end{equation}
	Here, for \(A\subset \rn\) the expression \(\text{int}(A)\) denotes the interior of \(A\). By classical properties of the signed distance, the map \(b_D\) is Lipschitz continuous. We observe that
	the intersections of any line segment with \(\p D\) correspond exactly to those points of the line segment that have zero distance function. Moreover,
	every line segment connecting two points from the interior to the exterior of \(D\) has at least one intersection with \(\p D\).
	
	After these premises, we argue by contradiction. Let \(x\in S\) and \(y\in D\) and assume that \([x,y]\) is not fully contained in \(D\). 
	This implies that there exists some \(z\in [x,y]\) such that \(z\notin D\). On the one hand, 
	the line segment \([x,y]\) must have at least two intersections with \(\p D\) because \(b_D\) has at least 
	one zero in \([x,z]\) and at least another one in \([z,y]\). On the other hand, because \(D\) is 
	starshaped with respect to \(S\) (in the sense of Definition \ref{def-ss-ii}) by assumption, the ray 
	starting in \(x\) and passing through \(y\) must have exactly one intersection with \(\p D\). This leads to a contradiction.
\end{proof}

As it turns out, in general, the notion of starshaped domains by means of condition (\ref{eq-ls}) is weaker than the one provided in Definition \ref{def-ss-ii}. The following counterexample illustrates this fact (compare also with Figure \ref{fig-counterex}). 

Let \(B_1\subset \R^2\) be the unit ball centered in the origin, and define \(D=B_1\backslash \{(0,y):\ y\in [\frac{1}{2},1)\}\) as the slit disk. Then, setting \(S:=\{(0,0)\}\), for every \((x,y)\in D\), the line segment \([(0,0),(x,y)]\) is completely contained in \(D\). On the other hand, the ray starting from \((0,0)\in S\) and going through the point \((0,1)\) has infinitely many intersections with \(\p D\) along the segment \(\{(0,y):\ y\in [1/2,1)\}\).

\begin{figure}[htbp] 
	\centering
	\includegraphics[width=12cm]{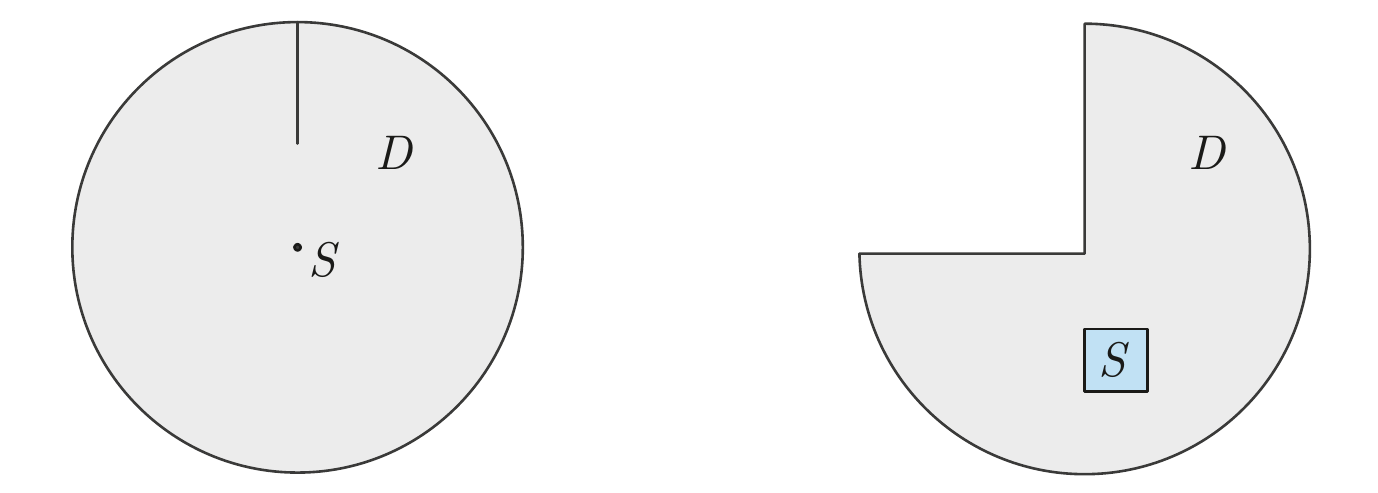}  
	\caption{Two counterexamples to the sufficiency of (\ref{eq-ls}). In the left picture \(D\) is the slit disk and \(S\) is just a point. In the right picture \(D\) is constructed by removing a whole quarter from the unit disk and \(S\) is a closed square. In both cases there exist rays starting at a point in \(S\) that have infinite intersections with the boundary of \(D\).}
	\label{fig-counterex}
\end{figure}

However, the situation can be rectified in the case in which \(S\) is open (and \(D\) is any domain). In this latter setting, property (\ref{eq-ls}) is also sufficient for \(D\) to be starshaped with respect to \(S\). 

\begin{lemma}[Sufficiency of condition (A)]\label{lem-sc}
	Let \(D\subset \rn\) be a domain and let \(S\subset D\) be an open set such that the pair \((D,S)\) fulfills property (\ref{eq-ls}). Then \(D\) is starshaped with respect to \(S\) in the sense of Definition \ref{def-ss-ii}. 
\end{lemma} 
\begin{proof}
	We argue by contradiction and assume that there exist \(y_1,y_2\in\p D\) with \(y_1\ne y_2\), and a segment starting in some \(x\in S\) such that it intersects \(\p D\) in at least these two points. Without loss of generality we assume that such segment coincides with $[x,y_2]$, that $[x,y_2]\cap \p D=\{y_1\}\cup\{y_2\}$, and that \(y_1\) is an interior point of $[x,y_2]$.
	
	The basic idea is that \(S\) being open allows perturbing the segment infinitesimally to force a contradiction with property (\ref{eq-ls}). Indeed, since $y_2 \in \partial D$, there exists some direction \(\sigma\in \sn\) and \(\veps>0\) such that \(y_3:=y_2+\veps\sigma\in D\) and such that the segment from \(y_3\) containing $y_1$ will meet a neighborhood of \(x\) fully contained in S. The fact that \(D\) is open guarantees that \(y_1\in \p D\) is not an element of \(D\) itself, leading to a violation of (\ref{eq-ls}).
\end{proof}

For the proof of Proposition \ref{prop-union-ss} we rely on the following auxiliary lemma.

\begin{lemma}[Cone property of Lipschitz domains]\label{lem-cone-prop}
	Let \(\Om\subset \rn\) be a bounded Lipschitz domain. Then, every \(x\in\Om\) is the vertex of a cone whose closure is contained in \(\Om\), meaning that for every \(x\in\Om\) there exist \(a,b>0\) and a rotation \(R\in SO(n)\) such that the closure of the cone 
	\begin{equation}
 \label{eq:def-cone}
	\mathcal{C}(a,b)
	:=\{y=(y_1,...,y_n)\in\rn:\ ay_n^2>y_1^2+\cdots +y_{n-1}^2 \text{\ and\ } 0<y_n<b\}
	\end{equation}
	is contained in \(\Om\) after rotating it by \(R\) and translating it to \(x\). Furthermore, \(a\) and \(b\) from above can be chosen uniformly for all \(x\in\Om\).
\end{lemma}

A detailed proof of the claim can be found for example in \cite[Theorem 3.3.1, p.91, and Corollary 3.3.2, p.96]{Fiorenza_2016}. Here we will only present the intuitive idea of the proof.
\begin{proof}
	Since \(\Om\) is bounded and the boundary \(\p \Om\) is of class \(C^{0,1}\), there exist \(\delta>0\), \(L>0\) such that inside the boundary region
	\begin{equation*}
	\Om_{\delta}
	:=\{x\in\Om:\ \dist(\p\Om,x)<\delta\}
	\end{equation*}
	the set \(\Om\) is locally (with respect to some finite open covering of the boundary \(\p\Om\)) the subgraph of a Lipschitz function with Lipschitz constant \(L\). From this, it follows that every point belonging to \(\Om_\delta\) is the vertex of a (suitably rotated) cone of the type \(\mathcal{C}(L',b_1)\)  whose closure is contained in \(\Om\), for some slope \(L'<L\)  and some uniform \(b_1>0\) (possibly depending on the dimension). Furthermore, there exist uniform \(a_2,b_2>0\) (again possibly depending on \(n\)) such that for every \(x\in\Om\backslash\Om_\delta\) all cones of the type \(C(a_2,b_2)\) whose vertex is \(x\) and with arbitrary orientation are such that their closure is contained in \(\Om\). This yields the claim.
\end{proof}

Now, any cone as in \eqref{eq:def-cone} is starshaped with respect to an open ball in the sense of Definition \ref{def-ss-i}. Thus, to ease the notation, in what follows \(\mathcal{C}_{x}\) will always denote a cone  contained in \(\overline{\Om}\) and starshaped with respect to a ball centered in \(x\).

\renewcommand{\proofname}{Proof of Proposition \ref{prop-union-ss}}

\begin{proof}
    We need to show that every bounded Lipschitz domain \(\Om\subset \R^n\) is the union of a finite number of domains which are
starshaped with respect to open balls in the sense of Definition \ref{def-ss-i} (or equivalently Definition \ref{def-ss-ii} since we established their equivalence above).
	We follow the proof strategy of \cite[Lemma 2, p.22]{Maz_ya_1998}. It follows from Lemma \ref{lem-cone-prop} that we can write \(\Om\) as
	the (possibly not finite) union of all points in the domain and the corresponding cones which can be chosen congruent. This implies in particular that there exists a set \(X\subset \Om\) such that 
	\begin{equation*}
	\Om
	= \cup_{x\in X} \mathcal{C}_x,
	\end{equation*}
	where all cones are starshaped with respect to open balls with the same radius \(R>0\).  Without loss of generality, we can assume that $X$ is a closed set. These cones yield a finite covering of \(\Om\) via the following construction:
	\begin{enumerate}
		\item
		Because \(\Om\) is bounded by assumption, the same holds true for \(X\subset \Om\), so that $X$ is compact. Hence, there exists a finite open cover of \(X\) with balls \(B_{R/2}(x_i)\) centered at \(x_i\in X\), where  \(1 \leqslant i  \leqslant N\) for some \(N\in\N\). 
		\item
		For every \(i\in \{1,..,N\}\) we define \(U_{i}:=B_{R/2}(x_i)\cap X\) and  
		\begin{equation*}
		\Om_i:= \cup_{x\in U_{i}} \mathcal{C}_x.
		\end{equation*}
		We note that \(B_{R/2}(x_i)\subset B_R(x)\subset \mathcal{C}_x\) for every \(x\in U_{i}\). It follows from either the Definition \ref{def-ss-i} or Definition \ref{def-ss-ii} that all \(\mathcal{C}_x\) are also starshaped with respect to \(B_{R/2}(x_i)\). Consequently, also the union \(\Om_i\) is starshaped with respect to \(B_{R/2}(x_i)\). 
	\end{enumerate}
	Altogether this guarantees the existence of sets \(\Om_i\), \(1 \leqslant i \leqslant N\), that are all starshaped with respect to open balls with  the same radius and fulfill
	\begin{equation*}
	\Om
	= \bigcup_{i=1}^N \Om_i.
	\end{equation*}
	This completes the proof. 
\end{proof}
\renewcommand{\proofname}{Proof}

We finally prove an auxiliary estimate for integrals appearing frequently throughout our computations. The key point in the next result is the fact that the set $\Omega$ has \emph{positive} finite measure.

\begin{lemma}\label{lem-aux-int-bounds}
	Let \(\Om\subset \R^n\) be a measurable set of positive finite measure, then for all \(z\in \R^n\) there holds
	\begin{equation}\label{eq-int-est}
		\int_\Om \frac{1}{|z-x|^{n-1}}\ {\dl x}
		 \leqslant (n+1)\omega_n^{1-\frac{1}{n}}|\Om|^{\frac{1}{n}},
	\end{equation}
	where \(\omega_n\) is the volume of the unit ball in \(\R^n\).
\end{lemma}
\begin{proof}
    For \(z\in \R^n\) we set \(R:= (|\Om|/\omega_n)^{1/n}>0\). Then, there holds
    \begin{equation*}
        |B_R(z)|
        =\Om
    \end{equation*}
    where \(B_R(z)\) is the ball centered at \(z\) with radius \(R\). From the equality
    \begin{equation*}
        |B_R(z)\cap \Om|+|B_R(z)\backslash\Om|
        =|B_R(z)|
        =|\Om|
        =|B_R(z)\cap \Om|+|\Om\backslash B_R(z)|
    \end{equation*}
    we further conclude that \(|\Om\backslash B_R(z)|=|B_R(z)\backslash\Om|\). It follows that 
    \begin{align*}
        \int_{\Om\backslash B_R(z)} \frac{1}{|z-x|^{n-1}}\ {\dl x}
        & \leqslant |\Om\backslash B_R(z)|\frac{1}{R^{n-1}}\\
        &= |B_R(z)\backslash\Om|\frac{1}{R^{n-1}}
         \leqslant 
        \int_{B_R(z)\backslash\Om} \frac{1}{|z-x|^{n-1}}\ {\dl x}.
    \end{align*}
    Therewith, for any \(z\in \R^n\) we can estimate the left-hand side of equation (\ref{eq-int-est}) by an integral over the ball \(B_R\) centered at the origin in the following way:
    \begin{align*}
        \int_\Om \frac{1}{|z-x|^{n-1}}\ {\dl x}
        &= \int_{\Om\backslash B_R(z)} \frac{1}{|z-x|^{n-1}}\ {\dl x}
        + \int_{B_R(z)\cap \Om} \frac{1}{|z-x|^{n-1}}\ {\dl x}\\
        & \leqslant \int_{B_R(z)\backslash \Om} \frac{1}{|z-x|^{n-1}}\ {\dl x}
        + \int_{B_R(z)\cap \Om} \frac{1}{|z-x|^{n-1}}\ {\dl x}\\
        &= \int_{ B_R(z)} \frac{1}{|z-x|^{n-1}}\ {\dl x}\\
        &= \int_{ B_R} \frac{1}{|x|^{n-1}}\ {\dl x}.
    \end{align*}
    By using that \((n+1)\omega_n=(n+1)|B_1|=|\p B_1|\) we finally conclude that
    \begin{equation*}
        \int_{B_R} \frac{1}{|x|^{n-1}}\ {\dl x}
        = \int_0^R \int_{\p B_r} \frac{1}{r^{n-1}}\ {\dl \sigma} {\dl r}
        = (n+1)\omega_n R
        = (n+1)\omega_n^{1-\frac{1}{n}}|\Om|^{\frac{1}{n}},
    \end{equation*}
    which proves the claimed inequality. 
\end{proof}

\smallskip

\section*{Acknowledgements}
E.D. acknowledges support
by the Austrian Science Fund (FWF) through projects 10.55776/F65, 10.55776/V662, 10.55776/Y1292, and 10.55776/P35359, as well as from OeAD through the WTZ grants CZ02/2022 and CZ 09/2023.
 G.Di~F. acknowledges support from  the Italian Ministry of Education and Research through the PRIN2022 project {\emph{Variational Analysis of Complex Systems in Material Science, Physics and Biology}} No.~2022HKBF5C. 
 G.Di~F. and L.H. acknowledge support from the Austrian Science Fund (FWF) through the project {\emph{Analysis and Modeling of Magnetic Skyrmions}} (grant 10.55776/P34609). G.Di~F. and L.H. thank the Hausdorff Research Institute for Mathematics in Bonn for its hospitality during the Trimester Program \emph{Mathematics for Complex Materials} funded by the Deutsche Forschungsgemeinschaft (DFG, German Research Foundation) under Germany Excellence Strategy – EXC-2047/1 – 390685813. G.Di~F. and L.H. also thank TU Wien and MedUni Wien for their support and hospitality.

\medskip 

\bibliographystyle{siam} 
\bibliography{master}

\begin{thebibliography}{10}

\bibitem{Acosta2004}
{\sc G.~Acosta and R.~G. Dur\'{a}n}, {\em An optimal {P}oincar\'{e} inequality
  in {$L^1$} for convex domains}, Proceedings of the American Mathematical
  Society, 132 (2004), pp.~195--202.

\bibitem{Attouch2014}
{\sc H.~Attouch, G.~Buttazzo, and G.~Michaille}, {\em Variational analysis in
  {S}obolev and {BV} spaces}, vol.~17 of MOS-SIAM Series on Optimization,
  Society for Industrial and Applied Mathematics (SIAM), Philadelphia, PA,
  second~ed., 2014.

\bibitem{Avkhadiev2007}
{\sc F.~G. Avkhadiev and K.-J. Wirths}, {\em Unified {P}oincar\'{e} and {H}ardy
  inequalities with sharp constants for convex domains}, ZAMM. Zeitschrift
  f\"{u}r Angewandte Mathematik und Mechanik. Journal of Applied Mathematics
  and Mechanics, 87 (2007), pp.~632--642.

\bibitem{Ayadi_2022}
{\sc H.~Ayadi, F.~Mokhtari, and R.~Souilah}, {\em The obstacle problem for
  noncoercive elliptic equations with variable growth and {$L^1$}-data},
  Portugaliae Mathematica, 79 (2022), pp.~61--83.

\bibitem{Bendahmane_2010}
{\sc M.~Bendahmane, P.~Wittbold, and A.~Zimmermann}, Journal of Differential
  Equations, 249 (2010), pp.~1483--1515.

\bibitem{Brezis2011}
{\sc H.~Brezis}, {\em Functional analysis, {S}obolev spaces and partial
  differential equations}, Universitext, Springer, New York, 2011.

\bibitem{Ciarlet_2011}
{\sc P.~G. Ciarlet and G.~Dinca}, {\em A poincar{\'{e}} inequality in a sobolev
  space with a variable exponent}, Chinese Annals of Mathematics, Series B, 32
  (2011), pp.~333--342.

\bibitem{CruzUribe2018}
{\sc D.~Cruz-Uribe, G.~Di~Fratta, and A.~Fiorenza}, {\em Modular inequalities
  for the maximal operator in variable {L}ebesgue spaces}, Nonlinear Analysis.
  Theory, Methods \& Applications. An International Multidisciplinary Journal,
  177 (2018), pp.~299--311.

\bibitem{CruzUribe2014}
{\sc D.~V. Cruz-Uribe and A.~Fiorenza}, {\em Variable Lebesgue Spaces},
  Springer Basel, 2013.

\bibitem{Dautray1990}
{\sc R.~Dautray and J.-L. Lions}, {\em Mathematical analysis and numerical
  methods for science and technology. {V}ol. 4}, Springer-Verlag, Berlin, 1990.

\bibitem{Di_Fratta_2022}
{\sc G.~Di~Fratta and A.~Fiorenza}, {\em A unified divergent approach to
  {H}ardy{\textendash}{P}oincar{\'{e}} inequalities in classical and variable
  {S}obolev spaces}, Journal of Functional Analysis, 283 (2022), p.~109552.

\bibitem{Fratta_2023}
{\sc G.~Di~Fratta, A.~Fiorenza, and V.~Slastikov}, {\em On symmetry of energy
  minimizing harmonic-type maps on cylindrical surfaces}, Mathematics in
  Engineering, 5 (2023), pp.~1--38.

\bibitem{DiFratta2019}
{\sc G.~Di~Fratta, V.~Slastikov, and A.~Zarnescu}, {\em On a sharp
  {P}oincar\'{e}-type inequality on the 2-sphere and its application in
  micromagnetics}, SIAM Journal on Mathematical Analysis, 51 (2019),
  pp.~3373--3387.

\bibitem{DiBenedetto_2016}
{\sc E.~DiBenedetto}, {\em Real Analysis}, Springer New York, 2016.

\bibitem{Diening2011}
{\sc L.~Diening, P.~Harjulehto, P.~H\"{a}st\"{o}, and M.~R{u}\v{z}i\v{c}ka},
  {\em Lebesgue and {S}obolev spaces with variable exponents}, vol.~2017 of
  Lecture Notes in Mathematics, Springer, Heidelberg, 2011.

\bibitem{Fan2005}
{\sc X.~Fan, Q.~Zhang, and D.~Zhao}, {\em Eigenvalues of {$p(x)$}-{L}aplacian
  {D}irichlet problem}, Journal of Mathematical Analysis and Applications, 302
  (2005), pp.~306--317.

\bibitem{Fiorenza_2016}
{\sc R.~Fiorenza}, {\em Hölder and locally Hölder Continuous Functions, and
  Open Sets of Class $C^k$, $C^{k,\lambda}$}, Springer International
  Publishing, 2016.

\bibitem{Gilbarg_2001}
{\sc D.~Gilbarg and N.~S. Trudinger}, {\em Elliptic Partial Differential
  Equations of Second Order}, Springer Berlin Heidelberg, 2001.

\bibitem{Kuznetsov2015}
{\sc N.~Kuznetsov and A.~Nazarov}, {\em Sharp constants in the {P}oincar\'{e},
  {S}teklov and related inequalities (a survey)}, Mathematika. A Journal of
  Pure and Applied Mathematics, 61 (2015), pp.~328--344.

\bibitem{Leoni2017}
{\sc G.~Leoni}, {\em A first course in {S}obolev spaces}, vol.~181 of Graduate
  Studies in Mathematics, American Mathematical Society, Providence, RI,
  second~ed., 2017.

\bibitem{Lieb2001}
{\sc E.~H. Lieb and M.~Loss}, {\em Analysis}, vol.~14 of Graduate Studies in
  Mathematics, American Mathematical Society, Providence, RI, second~ed., 2001.

\bibitem{MR2443740}
{\sc F.-Y. Maeda}, {\em Poincar\'{e} type inequalities for variable exponents},
  JIPAM. Journal of Inequalities in Pure and Applied Mathematics, 9 (2008),
  pp.~Article 68, 5.

\bibitem{Maz_ya_1998}
{\sc V.~G. Maz'ya and S.~V. Poborchi}, {\em Differentiable functions on bad
  domains}, World Scientific Publishing Co., Inc., River Edge, NJ, 1997.

\bibitem{Michlin1981}
{\sc S.~G. Mikhlin}, {\em Konstanten in einigen {U}ngleichungen der
  {A}nalysis}, vol.~35 of Teubner-Texte zur Mathematik [Teubner Texts in
  Mathematics], BSB B. G. Teubner Verlagsgesellschaft, Leipzig, 1981.
\newblock Translated from the Russian by Reinhard Lehmann, With English, French
  and Russian summaries.

\bibitem{ouaro2021structural}
{\sc S.~Ouaro and N.~Sawadogo}, {\em Structural stability for nonlinear
  $p(u)$-laplacian problem with fourier boundary condition}, Gulf Journal of
  Mathematics, 11 (2021), pp.~1--37.

\bibitem{Payne1960}
{\sc L.~E. Payne and H.~F. Weinberger}, {\em An optimal {P}oincar\'{e}
  inequality for convex domains}, Archive for Rational Mechanics and Analysis,
  5 (1960), pp.~286--292 (1960).

\bibitem{Ruzhansky2019}
{\sc M.~Ruzhansky and D.~Suragan}, {\em Hardy inequalities on homogeneous
  groups}, vol.~327 of Progress in Mathematics, Birkh\"{a}user/Springer, Cham,
  2019.
\newblock 100 years of Hardy inequalities.

\bibitem{Verfuerth1999}
{\sc R.~Verf\"{u}rth}, {\em A note on polynomial approximation in {S}obolev
  spaces}, M2AN. Mathematical Modelling and Numerical Analysis, 33 (1999),
  pp.~715--719.

\bibitem{Ziemer89}
{\sc W.~P. Ziemer}, {\em Weakly differentiable functions}, vol.~120 of Graduate
  Texts in Mathematics, Springer-Verlag, New York, 1989.
\newblock Sobolev spaces and functions of bounded variation.

\end{thebibliography}

\end{document}